\documentclass[12pt,reqno]{amsart}
\usepackage{latexsym,amsmath,amssymb,mathrsfs,setspace,pstricks,amsthm}
\pdfoutput=1
\newtheorem{theo}{Theorem}
\newtheorem{lemma}[theo]{Lemma}
\newtheorem{corollary}[theo]{Corollary}

\theoremstyle{remark}

\begin{document}
\doublespace

\title{ Units in $F_2D_{2p}$}

\author{ Kuldeep Kaur, Manju Khan}
\email{  kuldeepk@iitrpr.ac.in, manju@iitrpr.ac.in}

\address{Department of Mathematics, Indian Institute of Technology
Ropar,Nangal Road, Rupnagar - 140 001 }

\footnotetext{}


\maketitle

\markboth{ Kuldeep Kaur, Manju Khan}{Unit Group of $F_2D_{2p}$}

\doublespacing

\begin{abstract}
 Let $p$ be an odd prime, $D_{2p}$ be the dihedral group of order $2p$, and
$F_{2}$ be the finite field with two elements. If $*$ denotes the canonical
involution of the group algebra $F_2D_{2p}$, then bicyclic units are unitary
units. In this note, we investigate the structure of the group
$\mathcal{B}(F_2D_{2p})$, generated by the bicyclic units of the group algebra
$F_2D_{2p}$. Further, we obtain the structure of the unit group
$\mathcal{U}(F_2D_{2p})$ and the
unitary subgroup $\mathcal{U}_*(F_2D_{2p})$, and we prove that both
$\mathcal{B}(F_2D_{2p})$ and
$\mathcal{U}_*(F_2D_{2p})$ are normal subgroups of $\mathcal{U}(F_2D_{2p})$.
\end{abstract}
\section*{introduction}
 \noindent Let $FG$ be the group algebra of the group $G$ over the field $F$ and
 $\mathcal{U}(FG)$ denotes its unit group. The anti-automorphism $g \mapsto
 g^{-1}$ of $G$ can be extended linearly to an anti-automorphism $a
 \mapsto a^*$ of the group algebra $FG$ known as canonical involution of $FG$.
 Let $\mathcal{U}_*(FG)$ be the unitary subgroup
 consisting of the elements of $\mathcal{U}(FG)$ that are inverted by canonical
involution $*$. These
 elements are called unitary units in $FG$. If $F$ is a finite field of
 characteristic 2, then $\mathcal{U}_*(FG)$ coincides with $V_*(FG)$; otherwise,
 it coincides
 with $V_*(FG) \times \langle -1 \rangle$. Here $V_*(FG)$ denotes the set of
 all unitary units in the normalized unit group $V(FG)$. Interest in the group
$U_*(FG)$ arose in algebraic topology and unitary $K$-theory \cite{MR0292913}.  
 
 We are interested in the structure of the unit group $\mathcal{U}(F_2D_{2p})$
and the unitary subgroup $\mathcal{U}_*(F_2D_{2p})$. For a finite abelian
$p$-group $G$ and the field $F$ with $p$ elements, R. Sandling
\cite{MR761637} gave the structure of $V(FG)$ and the structure of $V_*(FG)$ was obtained by A.A.Bovdi and
 A.A.Sakach in \cite{FPG} for a finite field $F$ of
 characteristic $p$. For a field $F$ with two elements and a 2-group $G$
upto
order 16, R. Sandling \cite{MR1136226} gave the presentation for
$V(FG)$. 
 Later on, A.Bovdi and L. Erdei \cite{F2D8} described the structure of the
 unitary
 subgroup $V_*(F_2G)$, where $G$ is a nonabelian group of order 8 and 16. In
 \cite{order} V.
 Bovdi
 and A. L. Rosa computed the order of the unitary subgroup of the
 group
 of units when $G$ is either an extraspecial 2-group or the central product of
 such a group with a cyclic group of order 4, and $F$ is a finite field of
characteristic 2. In the same paper, they
 computed the order of the unitary subgroup $V_*(F G)$, where $G$ is a 2-group
 with an abelian subgroup $A$ of index 2 and an element $b$ such that $b$
 inverts every element in $A$ and the order of $b$ is 2 or 4. V. Bovdi and T.Rozgonyi in
 \cite{unitarybovdi} described the structure of $V_*(F_2G)$, where the order of $b$ is 4.

For a dihedral group $G$ of order $6 $ and $10$ and an arbitrary finite field $F$, the structure of the unit group $\mathcal{U}(FG)$ is described in
\cite{FS3} and \cite{FD10}. Here we give the structure of the unit group
 $\mathcal{U}(F_2D_{2p})$ and the group $\mathcal{U}_*(F_2D_{2p})$. The bicyclic
units of $F_2D_{2p}$ play an important role in finding the structure of unit
group.
 We also study the structure of the group, $\mathcal{B}(F_2D_{2p})$, generated
by bicyclic units of $F_2D_{2p}$.

For an element $g \in G$ of order $n$ , write $\widehat{g} = 1 + g + g^2 +
\cdots + g^{n -1} $. If $g, h \in G, o(g) < \infty$,  then \[u_{g, h} = 1 + (g -
1)h\widehat{g}\] has an inverse $u_{g, h}^{-1} = 1 - (g-1)h\widehat{g}$.
Moreover, $u_{g, h} = 1$ if and only if $h$ is in the normalizer of $\langle g
\rangle$.
The element $u_{g, h}$ is known as a bicyclic unit of the group algebra $FG$ and
the group generated by them is denoted by $\mathcal{B}(FG)$. Observe that all
nontrivial bicyclic units of the group algebra $F_2D_{2p}$ are
 unitary units with respect to canonical involution.

Let $F_q$ be a finite field with $q$ elements and $n$ be a positive integer
co-prime with $q$. If order of $q \ mod \ n $ is $d$, then the set $\{a_0, a_0q,
\ldots a_0q^{d-1}\}$
of elements of $\mathbb{Z}_n$ is said to be $q$-cycle modulo $n$. Further, if
$\alpha$ is a primitive $n$-th root of unity, then the polynomial \[f_{a_0}(x) =
(x -
\alpha^{a_0})(x - \alpha^{a_0q}) \cdots (x - \alpha^{a_0q^{d -1}}),\]
is an irreducible factor of $\phi_n (x)$ over $F_q$  is of degree $d$. Hence,
the number of
irreducible factors of $\phi_n(x)$ over $F_q$ is $\frac{\phi(n)}{d}$. Since $F_2$
is a field
with 2 elements and $D_{2p}$ is the dihedral group of order $2p$, it follows
that if order of $2 \ mod \ p $ is $d$, then the number of irreducible factors of
the cyclotomic polynomial $\phi_p(x)$ over $F_2$ is $\frac{\phi( p)}{d}$ and
each
irreducible factor is of degree $d$.    

\section*{Unit Group of $F_2D_{2p}$}
 \begin{theo}
  Let $ G$ be the dihedral group \[D_{2p} = \langle a, b \ | \ a^p = 1, b^2 = 1,
b^{-1} a b = a^{-1}\rangle.\]
Suppose $V = \langle 1 + \widehat{D_{2p}}\rangle$, where $\widehat{D_{2p}}$
 denotes the sum of all elements of $D_{2p}$. Then \[\mathcal{U}(F_2D_{2p})/V
\cong 
\begin{cases}
\underbrace{GL_2(F_{2^{\frac{d}{2}}}) \times GL_2(F_{2^{\frac{d}{2}}}) \cdots
\times
GL_2(F_{2^{\frac{d}{2}}})}_{{\frac{\phi(p)}{d}}\text{copies}}, & \text{if }
 d \text{ is even } \\
\underbrace{GL_2(F_{2^{d}}) \times GL_2(F_{2^{d}}) \cdots \times
GL_2(F_{2^{d}})}_{{\frac{\phi(p)}{2d}}\text{copies}}, & \text{if
} d \text{ is odd }.
\end{cases}\]\\ and hence $|\mathcal{U}(F_2D_{2p})|= 
\begin{cases}
2 ((2^d-1)(2^d-2^{\frac{d}{2}}))^{\frac{\phi(p)}{d}}, & \text {if } d \text{ is
even}\\
           2 ((2^{2d}-1)(2^{2d}-2^{d}))^{\frac{\phi(p)}{2d}}, & \text {if } d
\text{ is odd }
\end{cases}$\\

 \end{theo}
We need the following lemmas:
\begin{lemma}
 Let $p$ be an odd prime such that order of $2 \ mod \ p $ is $d$. If $\zeta $
is a
primitive $p$-th root of unity, then $\zeta$ and $\zeta^{-1}$ are the roots of
the same irreducible factor of $\phi_p(x)$ over $F_2$ if and only if $d$ is
even.
\end{lemma}
\begin{proof}
Assume that $\zeta$ and $\zeta^{-1}$ are the roots of the same irreducible
factor of $\phi_p(x)$ over $F_2$. It follows that $-1$ and $1$ are in same
$2$-cycle $ \ mod \ p$ and so there exist some $t < d$ such that $ 2^{t}$ $
\equiv$  $-1 \mod p $. 
Hence order of $2^t  \mod p $ is 2. Further, since $2^d
\equiv 1 \mod p$, it implies that $(2^t)^{d} \equiv 1 \mod p$.
Hence $2|d$.\\
 Conversely, let $d$ be even, say $d= 2t$. Then $2^t \equiv -1 \mod
p $ and hence $-1$ and $1$ are in same $2$-cycle  $\ mod \ p$. The result follows.
\end{proof}

\begin{lemma}
 Let $\zeta$ be a primitive $p$-th root of unity. If order of $2 \mod p$ is $d$,
then $[F_2(\zeta + \zeta^{-1}): F_2] =
\frac{d}{2}$, if $d$ is even and $[F_2(\zeta + \zeta^{-1}): F_2] =
d$, if $d$ is odd and in this case $F_2(\zeta + \zeta^{-1})=F_2(\zeta)$.
\end{lemma}

\begin{proof}
We claim that $[F_2(\zeta):F_2(\zeta + \zeta^{-1})]= 1 \ \text{or} \ 2.$ If
$[F_2(\zeta):F_2(\zeta + \zeta^{-1})]= s > 2$, then the
degree of the minimal polynomial of $(\zeta + \zeta^{-1})$ over $F_2$
is  $\frac{d}{s}$, which is less than $\frac{d}{2}$. Hence, there is a polynomial
over $F_2$ satisfied by $\zeta$ of degree less than $d$, which is impossible.

 Now, if $d$ is even, then by last lemma, we obtain a polynomial of degree $d$
satisfied by $\zeta$ and $\zeta^{-1}$. It implies that there is a polynomial of
degree $d- 1$ satisfied by $\zeta + \zeta^{-1}$. Hence $[F_2(\zeta + \zeta^{-1})
:
F_2] < d$ and therefore,  $[F_2(\zeta) : F_2(\zeta + \zeta^{-1})] = 2$.
Further, if $d$ is odd, then $[F_2(\zeta):F_2(\zeta + \zeta^{-1})] \neq 2$.
Hence 
$F_2(\zeta)=F_2(\zeta + \zeta^{-1})$ and so $[F_2(\zeta + \zeta^{-1}): F_2] =
d$.
 \end{proof}
\noindent { \it{ Proof of the Theorem.}}
Let the cyclotomic polynomial \[\phi_p(x) = f_1(x)f_2(x)\ldots f_{s}(x)\] be the
product of irreducible factors over $F_2$, where $s = \frac{\phi(p)}{d}$. 
 Assume that $\gamma_i$ is a root of irreducible factor $f_i(x)$ over $F_2$.
Define a matrix representation of $D_{2p}$,  \[T_{\gamma_i} : D_{2p} \rightarrow
M_2(F_2(\gamma_i+ \gamma_i ^{-1}))\] by the assignment

\[a \mapsto  \left(
\begin{array}{cc}
  0 & 1 \\
 1  & \gamma_i + \gamma_i^{-1} \\
\end{array}
\right), \ \ \  b \mapsto  \left(
\begin{array}{cc}
  1 & 0 \\
 \gamma_i + \gamma_i^{-1} & 1 \\
\end{array}
\right)\]
 If $d$ is even, then define $T = T_0 \oplus T_{\gamma_1} \oplus T_{\gamma_2}
\oplus \cdots \oplus T_{\gamma_s} $, the direct sum of the given representations
$T_{\gamma_i}, 1 \leq i\leq s$, and $T_0$ is the trivial representation of
$D_{2p}$ over $F_2$ of degree 1. 

\noindent Suppose $d$ is odd. Lemma (2) implies that $\gamma_i$ and
$\gamma_i^{-1}$ are roots of the different irreducible factors of
$\phi_p(x)$. If $\gamma_i^{-1}$ is a root of $f_j(x)$, then choose $\gamma_j =
\gamma_i^{-1}$. Without loss of generality, assume that $\gamma_1, \gamma_2,
\ldots \gamma_{s'}$ are the roots of distinct irreducible factors of $\phi_p(x)$ such that $\gamma_i \neq
\gamma_j^{-1}$ for $1 \leq i, j \leq s'$.
 Then define $T = \displaystyle T_0 \mathop{\oplus}_{i = 1}^{s'} T_{\gamma_i}$,
the
direct sum of all
distinct matrix representation. 
Therefore, 
\[T: D_{2p} \rightarrow \mathcal{U}(F_2 \oplus M_2( F_2(\gamma_1 +
\gamma_1^{-1})) \oplus
\cdots \oplus M_2( F_2(\gamma_k + \gamma_k^{-1}))) \]
given
by 
 \[a \mapsto \Bigg(1, \left(
\begin{array}{cc}
  0 & 1 \\
 1  & \gamma_1 + \gamma_1^{-1} \\
\end{array}
\right),  \cdots ,\left(
\begin{array}{cc}
  0 & 1 \\
 1  & \gamma_k + \gamma_k^{-1} \\
\end{array}
\right)\Bigg)\ \ \\ \]
 $\mbox{and}\ \ $ \[b \mapsto \Bigg(1, \left(
\begin{array}{cc}
  1 & 0 \\
 \gamma_1 + \gamma_1^{-1} & 1 \\
\end{array}
\right),  \cdots , \left(
\begin{array}{cc}
  1 & 0 \\
 \gamma_k + \gamma_k^{-1} & 1 \\
\end{array}
\right)\Bigg)\] \\
is a group homomorphism, where $k = s = \frac{\phi(p)}{d}$, if $d$ is even and $k = s' = \frac{\phi(p)}{2d}$ if $d$ is odd. 

Extend this group homomorphism $T$ to the algebra
homomorphism 
\[T': F_2D_{2p} \rightarrow F_2 \oplus M_2( F_2(\gamma_1 + \gamma_1^{-1}))
\oplus
\cdots \oplus M_2( F_2(\gamma_k + \gamma_k^{-1})), \] where $M_2(F_2(\gamma_i +
\gamma_i^{-1}))$ is the algebra of $2 \times 2$ matrices over the field
$F_2(\gamma_i + \gamma_i^{-1})$.

\noindent Note that the representation $T_{\gamma_i}$ is equivalent to
$S_{\gamma_i}$,
where
\[S_{\gamma_{i}}(a)= \left(
\begin{array}{cc}
  \gamma_{i} & 0 \\
 0  & \gamma_i^{-1} \\
\end{array}
\right),   \ \ S_{\gamma_{i}}(b)= \left(
\begin{array}{cc}
  0 & 1 \\
 1  & 0 \\
\end{array}
\right).\]
Hence, for $x \in D_{2p}, T_{\gamma_i}(x) = M_i S_{\gamma_i}(x)M_i^{-1}$, where
$M_{i}=\left(
\begin{array}{cc}
  1 & 1 \\
  \gamma_i &   \gamma_i^{-1} \\
\end{array}
\right)$. 
Suppose that $x = \displaystyle \sum_{i= 0}^{p-1}\alpha_{i} a^i +  \sum_{i =
0}^{p -1}\beta_{i}a^ib \in Ker T'$. 
 Then $T'(x)=0$ implies that
\begin{equation}
 \sum_{i=0}^{p-1}\alpha _{i}+
\sum_{i=0}^{p-1}\beta _{i}=0
\end{equation}
 and 
 for $1 \leq j \leq k, \gamma_{j}$ and $\gamma_{j}^{-1}$ satisfies the
polynomials $g(x)=\alpha_{0}+\alpha_{1}x+ \cdots + \alpha_{p-1}x^{p-1}$ and
 $h(x)=\beta_{0}+\beta_{1}x+ \cdots + \beta_{p-1}x^{p-1}$ over $F_{2}$. 
It follows that  irreducible factors of $\phi_p(x)$ are factors of $g(x)$
and $h(x)$. Further, since all factors are co-prime, it follows that
$\phi_{p}(x)$ divides $g(x) \ \text{and} \  h(x)$ and hence $\alpha_i =
\alpha_j, \text{and}  \ \beta_i =
\beta_j , 0 \leq i, j \leq {p-1}$.
Thus, from equation $(1)$, we have 
$\alpha_i = \beta_i, 0 \leq i \leq p-1$ and therefore $Ker T' =
F_2\widehat{D_{2p}}$. 

Further, the dimension of $(F_2D_{2p}/F_2\widehat{D_{2p}})$ and $  \displaystyle
F_2 \mathop{\oplus}_{i=1}^{k} M_2( F_2(\gamma_i +
\gamma_i^{-1}))$ over $F_2$ are same. Hence
\[F_2D_{2p}/F_2\widehat{D_{2p}} \cong \displaystyle F_2
\mathop{\oplus}_{i=1}^{k} M_2(
F_2(\gamma_i +
\gamma_i^{-1})).\] Since $F_2\widehat{D_{2p}}$
is nilpotent, $T'$ induces an epimorphism
\[T'':\mathcal{U}(F_2D_{2p}) \rightarrow \prod_{i=1}^{k} GL_2( F_2(\gamma_i +
\gamma_i^{-1}))\] such that $ker T'' =
\langle 1+\widehat{D_{2p}}\rangle$. Hence 
\[\mathcal{U}(F_2D_{2p})/\langle 1+\widehat{D_{2p}}\rangle \cong \prod_{i=1}^{k}
 GL_2(F_2(\gamma_i + \gamma_i^{-1}))\]
  and therefore the result follows.
\section*{Structure of $\mathcal{B}(F_2D_{2p})$}

\begin{theo}
 Let p be an odd prime such that order of $ 2 \mod p \text{ is } d$. Then, the group
generated by the bicyclic units, i.e.,
 \[\mathcal{B}(F_2D_{2p}) \cong 
 \begin{cases}
\underbrace{SL_2(F_{2^{\frac{d}{2}}}) \times SL_2(F_{2^{\frac{d}{2}}})\cdots
\times
SL_2(F_{2^{\frac{d}{2}}})}_{{\frac{\phi(p)}{d} }\text{copies}}, & \text{if } d
\text{ is even} \\
\underbrace{SL_2(F_{2^{d}}) \times SL_2(F_{2^{d}}) \cdots \times
SL_2(F_{2^{d}})}_{{\frac{\phi(p)}{2d}}\text{copies}}, & \text{if
} d \text{ is odd}
\end{cases} \] where
  $SL_2(F)$ is the special linear group of degree 2 over $F$.
\end{theo}
 We need the following lemmas:
\begin{lemma}
$ D_{2p} \cap \mathcal{B}(F_2D_{2p})= \langle a \rangle$.
\end{lemma}

\begin{proof}
  Since $D_{2p}$ is in the normalizer of $\mathcal{B}(F_2D_{2p})$, it implies
that
 $\mathcal{B}(F_2D_{2p}) \cap D_{2p}$ is a normal subgroup of $D_{2p}$.
 Therefore, it is either a trivial subgroup or $\langle a \rangle$.\\
We claim that $ b \notin D_{2p} \cap \mathcal{B}(F_2D_{2p})$. For that we
  define a map \[f : D_{2p} \rightarrow  \langle g \ | \ g^2 = 1\rangle\] such
that
  \[f(a^i) = 1 \ \textrm{and} \ f(a^i b) = g , 0 \leq i \leq p -1.\] Note that
 it is a group homomorphism and  we can extend this linearly to an algebra homomorphism
 $f'$ from $F_2D_{2p}$ to $F_2 \langle g \rangle$. It is easy to see that the
image of the bicyclic units under $f'$ is 1. If $b \in
 \mathcal{B}(F_2D_{2p})$, then $f'(b)= 1$, which is not possible.
 Therefore, $b \notin \mathcal{B}(F_2D_{2p})$.
This shows that $D_{2p} \cap \mathcal{B}(F_2D_{2p}) \neq D_{2p}$.\\
Also observe that \[u_{ab, a}u_{ab, a^2} \cdots u_{ab,a^l}=
ab(1 + \widehat{D_{2p}})\]  \[\text{and}\ \ u_{b, a}u_{b, a^2} \cdots u_{b,
a^l}=b(1 +
\widehat{D_{2p}}),\]
 where $u_{a^jb, a^i} = 1 + (a^i + a^{-i})(1 + a^jb)$ is a bicyclic unit of
the group algebra $F_2D_{2p}$ and $l = \frac{p-1}{2}$. It implies that
\[a=u_{ab, a}u_{ab, a^2} \cdots
u_{ab,a^l} u_{b, a^l} \cdots u_{b, a^2}u_{b, a}.\] Hence,
  $ D_{2p} \cap
\mathcal{B}(F_2D_{2p})= \langle a \rangle$.
\end{proof}

\begin{lemma}
 For  $1 \leq i \leq k$, let $\gamma_i $  be the primitive $p$-th root of unity
described in  the proof of the Theorem 1. Then the minimal polynomials of
$\gamma_i +\gamma_i^{-1}, 1 \leq i \leq k$,  are distinct.
\end{lemma}
\begin{proof}
Suppose that $d$ is even. If $\zeta$ is a primitive $p$-th root of unity, then $
[F_2(\zeta  + \zeta^{-1}): F_2 ] = \frac{d}{2}$. Assume that 
\[f(x) = a_0 + a_1x + \cdots + a_{\frac{d}{2}}x^{\frac{d}{2}} \] is the minimal
polynomial over $F_2$ satisfied by both $\gamma_i + \gamma_i^{-1}$ and $\gamma_j
+ \gamma_j^{-1}$ for $i \neq j.$ It implies that there is a polynomial of degree
$d$ over $F_2$ satisfied by both $\gamma_i$ and $\gamma_j$. This is a
contradiction, because the minimal polynomials of $\gamma_i$ and $\gamma_j$ over
$F_2$ are co-prime. Hence the result follows.

\noindent Further, if $d$ is odd, then $[F_2(\zeta + \zeta^{-1}):F_2] = d$. Let
$f(x)$ be the minimal polynomial over $F_2$ satisfied by both $\gamma_i +
\gamma_i^{-1}$ and $\gamma_j + \gamma_j^{-1}$ for $i \neq j$. It follows that
there is a polynomial $g(x)$ of degree $2d$ over $F_2$ satisfied by $\gamma_i,
\gamma_j, \gamma_j$ and $\gamma_j^{-1}$. Since $d$ is odd, the minimal
polynomials of $\gamma_i^{\pm 1}$ and $\gamma_j^{\pm 1}$over $F_2$  are
co-prime.
Hence the product of the minimal polynomials divides $g(x)$, which is a contradiction.
This completes the proof of the lemma.
\end{proof}
\noindent \textbf{Proof of the theorem}: Observe that the image of the bicyclic
units of
the group algebra $F_2D_{2p}$  are in $\displaystyle \prod_{i = 1}^k
SL_2(F_2(\gamma_i +
\gamma_i^{-1}))$ under the map $T''$. Suppose $T'''$ is the restricted map of
$T''$ to $\mathcal{B}(F_2D_{2p})$,  
 i.e.,
\[T''':\mathcal{B}(F_2D_{2p}) \rightarrow 
\prod_{i = 1}^k SL_2(F_2(\gamma_i + \gamma_i^{-1}))\] such that
$T'''(x)=T''(x)$ for $ x \in \mathcal{B}(F_2D_{2p})$. Then $kerT''' \leq
kerT''$. Since $b \notin \mathcal{B}(F_2D_{2p})$ and $ u_{b, a}u_{b, a^2} \cdots
u_{b, a^l} = b(1 + \widehat{D_{2p}})$, it follows that $ kerT'''=\{ 1 \}$.

 \noindent Further, it is known that \[ SL_2(F_2(\zeta+\zeta^{-1})) = \Bigg
\langle \left(
 \begin{array}{cc}
   1 & u \\
  0 & 1 \\
\end{array}
 \right) 
  , \ \ \left(
 \begin{array}{cc}
   1 & 0 \\
  v & 1 \\
 \end{array}
 \right) \ \Bigg | \  u, v \in F_2(\zeta+\zeta^{-1})  \Bigg \rangle .\] \\To
prove $T'''$ is
onto, it is
sufficient to prove that the elements of the form 
\[ \Bigg(\left(
\begin{array}{cc}
  1 & 0 \\
 0  & 1 \\
\end{array}
\right), \cdots, \left(
\begin{array}{cc}
  1 & 0 \\
 u_i  & 1 \\
\end{array}
\right), \cdots \left(
\begin{array}{cc}
  1 & 0 \\
 0  & 1 \\
\end{array}
\right) \Bigg) \\ \] \\ and
\[ \Bigg(\left(
\begin{array}{cc}
  1 & 0 \\
 0  & 1 \\
\end{array}
\right), \cdots , \left(
\begin{array}{cc}
  1 & v_i \\
 0  & 1 \\
\end{array}
\right), \cdots \left(
\begin{array}{cc}
  1 & 0 \\
 0  & 1 \\
\end{array}
\right) \Bigg) \\ \]   \\
where $u_i,\ v_i \in F_2(\gamma_i + \gamma_i^{-1})$ have a preimage in
$\mathcal{B}(F_2D_{2p})$ under $T'''$  for all $1 \leq i \leq k $.

Assume that $y_i= \displaystyle \prod_{\stackrel{j=1}{j \neq
i}}^kf_j'(\gamma_i+\gamma_i^{-1})$, such that $f_j'(x)$ is the minimal
polynomial of $\gamma_j + \gamma_j^{-1}$ over $F_2$. If $g(x)=\displaystyle
\prod_{\stackrel{j=1}{j \neq
i}}^kf_j'(x)$, then $g(\gamma_j+\gamma_j^{-1})=0$ for $1 \leq j \leq k, \ j
\neq i$ and $g(\gamma_i+\gamma_i^{-1})= y_i$, a nonzero element of
$F_2(\gamma_i+\gamma_i^{-1})$.

Take $\{y_i, y_i(\gamma_i+\gamma_i^{-1}), \cdots ,
y_i(\gamma_i+\gamma_i^{-1})^{t-1}\}$
as a basis of $F_2(\gamma_i+\gamma_i^{-1})$ over $F_2$, where
$t=[F_2(\gamma_i+\gamma_i^{-1}):F_2]$. Therefore, any element $u_i $  of
$F_2(\gamma_i+\gamma_i^{-1})$ can be written as $\displaystyle u_i= y_i\sum_{j = 0}^{t
-1}\alpha_j (\gamma_i+\gamma_i^{-1})^{j}$. Assume that $u'(x) = g(x)u(x)$, where
$u(x)=\alpha_0+\alpha_1x+ \cdots +\alpha_{t-1} x^{t-1} \in F_2[x]$ and therefore $u'(x) \in F_2[x]$. It is clear that
$u'(\gamma_i+\gamma_i^{-1})=u_i$ and $u'(\gamma_j+\gamma_j^{-1})=0$ for $1 \leq
j \leq k,\ j \neq i$. Further, if $u'(x)= a_0 + a_1 x+ \cdots + a_m x^{m}$, then the generator $X_i$, whose i-th
component is $\left(
\begin{array}{cc}
  1 & 0 \\
 u_i  & 1 \\
\end{array}
\right) $ and other components are the identity matrix, can be written as $X_i=
e_0e_1 \cdots e_m$. Here $e_j$ is an element of $\displaystyle \prod_{i = 1}^kSL_2(F_2(\gamma_i + \gamma_i^{-1}))$ such that the $r$-th component of $e_j$ is $\left(
\begin{array}{cc}
  1 & 0 \\
 a_j(\gamma_r + \gamma_r^{-1})^j & 1 \\
\end{array}
\right)$, for $0 \leq j \leq m, \ 1 \leq r \leq k$.
Now we will prove that the preimage of $e_j$ is in $\mathcal{B}(F_2D_{2p})$
under the map $T'''$.

If $a_j=0$, then it is trivial. Now assume that $a_j=1$.
Suppose that $M = (M_1, M_2, \ldots M_k)$, where $M_r= \left(
\begin{array}{cc}
  1 & 1 \\
 \gamma_r & \gamma_r^{-1} \\
\end{array}
\right) $. If \[(\gamma_r + \gamma_r^{-1})^{j-1} = b_0 + \sum_{s = 1}^{l-1}
b_s(\gamma_r^s + \gamma_r^{-s} ), \] where $b_i \in F_2$  then the $r$-th component of $M^{-1}e_jM$  
is \[\left(
\begin{array}{cc}
  1 & 0 \\
 0 & 1 \\
\end{array}
\right) + (b_0 + \sum_{s = 1}^{l-1} b_s(\gamma_r^s + \gamma_r^{-s} ))\left(
\begin{array}{cc}
  1 & 1 \\
 1 & 1 \\
\end{array}
\right).\]
 By extending the matrix representation $S_{\gamma_r}$ to an algebra
homomorphism
over $F_2$, we obtain that this element is an image of $\alpha$ under the algebra
homomorphism $S_{\gamma_r}$, where $ \displaystyle \alpha = 1 + (b_0 + \sum_{s =
1}^{l-1} b_s(a^s + a^{-s} ))(1+b)  $. 
Since $S_{\gamma_r}(\alpha) =
M_r^{-1}T_{\gamma_r}(\alpha)M_r$, it follows that $e_j = T''(\alpha)$.
If $b_0 = 0$, then  \[\alpha = \prod_{s = 1}^{l-1} (1 + b_s(a^s + a^{-s})(1 +
b)),
\] is product of bicyclic units of the group algebra $F_2D_{2p}$. Now if $b_0 =
1$, then \[\alpha = b \prod_{s = 1}^{l-1} (1 + b_s(a^s + a^{-s})(1 + b)).\]
Since $b (1 + \widehat{D_{2p}})= u_{b, a}\ldots u_{b, a^l}$ and $1 +
\widehat{D_{2p}}$ is in the kernel of $T''$, it implies that $\alpha(1 +
\widehat{D_{2p}})$ is the preimage of $e_j$ under the map $T'''$. 

 Therefore, the preimage of $X_i$ is in
$\mathcal{B}(F_2D_{2p})$. Similarly we can prove the same thing for other generators. 
Then \[\displaystyle \prod_{i = 1}^k
SL_2(F_2(\gamma_i +
\gamma_i^{-1})) =
T'''(\mathcal{B}(F_2D_{2p})).\]
Hence 
\[\mathcal{B}(F_2D_{2p}) \cong
\displaystyle \prod_{i = 1}^k
SL_2(F_2(\gamma_i +
\gamma_i^{-1}))\]  and so \[|\mathcal{B}(F_2D_{2p})| 
= \begin{cases}
(2^{\frac{d}{2}}(2^{d}-1))^k & \text{if } d \text{ is even }\\
(2^d(2^{2d}-1))^k & \text{if } d \text{ is odd}.
\end{cases}\]\\

\section*{The structure of Unitary Subgroup and Unit Group}
\begin{theo} 
  The unitary subgroup $\mathcal{U}_*(F_2D_{2p})$ of the group algebra
 $F_2D_{2p}$ is the semidirect product of the normal subgroup
 $\mathcal{B}(F_2D_{2p})$ with the group $\langle b \rangle$. Further,
  $\displaystyle \mathcal{U}(F_2D_{2p})=\mathcal{U}_*(F_2D_{2p}) \times \prod_{i
= 1}^k \langle z_i \rangle
 $, where $ z_i$ is an invertible element in the 
 center of the group algebra $F_2D_{2p}$ of order $2^{\frac{d}{2}}-1$, if $d$ is
even; otherwise it is of order $2^d-1$ .
 \end{theo}
\begin{proof}
Since $GL_2(F_2(\gamma + \gamma^{-1}))$ is the direct product of $ SL_2(F_2(\gamma
+ \gamma^{-1}))$  with the group
 consisting of all nonzero scalar matrices, we have
\[\mathcal{U}(F_2D_{2p})/V \cong 
\prod_{i=1}^{k}(SL_2(F_2(\gamma_i+\gamma_i^{-1}))\times
(F_2(\gamma_i+\gamma_i^{-1}))^*I_{2 \times 2}), \]
where $F_2(\gamma+ \gamma^{-1})^*$ is the group of all nonzero elements of
$F_2(\gamma + \gamma^{-1})$. Let $F_2(\gamma_i + \gamma_i^{-1})^* = \langle
\eta_i \rangle$ for $1 \leq i \leq k$. Since $\displaystyle y_i =
\prod_{\stackrel{j = 1}{
j \neq i}}^k f_j'(\gamma_i + \gamma_i^{-1})$ is a non zero element of
$F_2(\gamma_i + \gamma_i^{-1})$, take $\{y_i, y_i(\gamma_i + \gamma_i^{-1}),
\ldots y_i(\gamma_i + \gamma_i^{-1})^{t -1} \}$ as a basis of $F_2(\gamma_i +
\gamma_i^{-1})$ over $F_2$. Here $t = \frac{d}{2}$ when $d$ is even; otherwise
$t = d$. Therefore, $\eta_i = y_ih_i(\gamma_i + \gamma_i^{-1}) = h_i'(\gamma_i +
\gamma_i^{-1})$, where $h_i(x) \ \text{and} \ h_i'(x) \in F_2[x]$. Also note
that
$h_i'(\gamma_j + \gamma_j^{-1}) = 0$ for $i \neq j$. If the constant coefficient
of $h_i'(x)$ is $\alpha_{\eta_i}$, then the image of $h_i'(a + a^{-1})$ under
the map $T'$ is the element $x_i'$ such that the first component of $x_i'$ is
$\alpha_{\eta_i}$, $(i+1)$-th component is $\left(
\begin{array}{cc}
  \eta_i & 0 \\
 0  & \eta_i \\
\end{array}
\right)$ and all the remaining components are zero matrix. Further, if
$\displaystyle y_0(x) = \prod_{i = 1}^kf_i'(x)$, then $y_0(\gamma_i +
\gamma_i^{-1}) = 0,
1 \leq i \leq k$ and the constant coefficient of $y_0(x) \ \text{is} \ 1$. It
follows that the image of $y_0(a + a^{-1})$ under the map $T'$ is the element
whose first component is $1$ and remaining components are zero. Choose $x_i$
such that $ (i + 1)$-th component is $\left(
\begin{array}{cc}
  \eta_i & 0 \\
 0  & \eta_i \\
\end{array}
\right)$ and the remaining components are identity matrix.
If $z_i$ denotes a preimage of $x_i$, then either \[z_i  = \displaystyle \sum_{\stackrel{j = 1}{j \neq
i}}^k
h_j'(a + a^{-1})^{2^t -1} +  h_i'(a + a^{-1})\] or \[z_i = \displaystyle
\sum_{\stackrel{j = 1}{j
\neq i}}^k h_j'(a + a^{-1})^{2^t -1} +  h_i'(a + a^{-1}) + y_0(a + a^{-1})\]
which
are in the center of $\mathcal{U}(F_2D_{2p})$ of order $2^t -1$. Further, since $\langle
z_i \rangle \cap \displaystyle \langle  
z_j \ | \ 1 \leq j \leq k, j \neq i\rangle = \{1\} $, take $W = \displaystyle
\prod_{i = 1}^k \langle
z_i \rangle$. Also, note that $W \cap \mathcal{U}_*(F_2D_{2p}) = \{1 \}$ and
therefore $\mathcal{U}(F_2D_{2p}) = W \times (\mathcal{B}(F_2D_{2p}) \rtimes
\langle b \rangle)$ and $U_{*}(F_2D_{2p})= \mathcal{B}(F_2D_{2p}) \rtimes
\langle b \rangle$.
\end{proof}
 
\begin{corollary}
The group generated by bicyclic units $\mathcal{B}(F_2D_{2p})$
 and the unitary subgroup $\mathcal{U}_*(F_2D_{2p})$ are
normal subgroups of
$\mathcal{U}(F_2D_{2p})$.
\end{corollary}

 \begin{corollary}
 The commutator subgroup $\mathcal{U}'(F_2D_{2p})=\mathcal{U}_*'(F_2D_{2p})$.
Also, $\mathcal{U}'(F_2D_{2p})$ is a normal subgroup of $\mathcal{B}(F_2D_{2p})$.
 \end{corollary}
 \begin{proof}
Since $\mathcal{U}(F_2D_{2p})= W \times \mathcal{U}_*(F_2D_{2p})$ such that $W$
is in the center of $F_2D_{2p}$, it follows that $\mathcal{U}'(F_2D_{2p}) =
\mathcal{U}_*'(F_2D_{2p})$ 
Further, since $\mathcal{U}_*(F_2D_{2p}) = \mathcal{B}(F_2D_{2p}) \rtimes
\langle b
 \rangle$
  and $ b $ is in the normalizer of $\mathcal{B}(F_2D_{2p})$, it implies that 
$\mathcal{U}_*'(F_2D_{2p}) \leq \mathcal{B}(F_2D_{2p}) \leq
 \mathcal{U}_*(F_2D_{2p}) $ and hence the result follows.

  \end{proof}

 \bibliographystyle{plain}
 \bibliography{UF2D2p}

\end{document}